\documentclass[10pt, reqno]{amsart}

\usepackage[T1]{fontenc}
\usepackage[utf8]{inputenc}
\usepackage{lmodern}

\usepackage{amsmath,amssymb,amsfonts,amsthm,mathtools}
\usepackage[mathscr]{eucal}

\usepackage{graphicx}
\usepackage{booktabs}
\usepackage{enumerate}
\usepackage{comment}
\usepackage{float}
\usepackage{subcaption}
\usepackage{cancel}
\usepackage{nicematrix}
\usepackage{systeme}

\usepackage{tikz,tikz-cd}
\usetikzlibrary{calc}
\usetikzlibrary{decorations.markings}
\usetikzlibrary{decorations.pathreplacing}

\usepackage[colorlinks=true,
            linkcolor=blue,
            citecolor=blue,
            urlcolor=blue,
            bookmarks=false]{hyperref}
\usepackage[nameinlink,capitalize]{cleveref}

\allowdisplaybreaks
\numberwithin{equation}{section}
\numberwithin{figure}{section}

\newtheorem{theorem}{Theorem}[section]
\newtheorem{proposition}[theorem]{Proposition}

\theoremstyle{definition}

\newtheorem*{Examples}{Examples}
\newtheorem*{Example}{Example}

\theoremstyle{remark}
\newtheorem{remark}[theorem]{Remark}

\newtheorem*{Remark}{Remark}

\newenvironment{entry-cor}[1][Cor]{%
  \begin{trivlist}
  \item[\hskip\labelsep{\normalfont\scshape #1}]%
}{\end{trivlist}}

\usepackage{etoolbox}

\makeatletter
\patchcmd\@thm
  {\let\thm@indent\indent}
  {\let\thm@indent\noindent}
  {}{}
\makeatother

\makeatletter
\renewenvironment{proof}[1][\proofname]{%
  \par
  \pushQED{\qed}%
  \normalfont \topsep6\p@\@plus6\p@\relax
  \trivlist
  \item[\hskip\labelsep\scshape #1\@addpunct{.}]%
  \ignorespaces
}{%
  \popQED\endtrivlist\@endpefalse
}
\makeatother

\theoremstyle{plain}

\newcommand{\entryheading}{}
\newtheorem*{entrybase}{\entryheading}

\newcounter{entryref}

\newenvironment{entry}[1]
  {%
    \renewcommand{\entryheading}{Entry~#1}%
    \begin{entrybase}%
    \renewcommand{\theentryref}{#1}%
    \refstepcounter{entryref}%
  }
  {\end{entrybase}}

\newenvironment{corentry}[2][]
  {%
    \renewcommand{\entryheading}{%
      Corollary%
      \if\relax\detokenize{#1}\relax
      \else
        ~(#1)%
      \fi
      \if\relax\detokenize{#2}\relax
      \else
        ~to Entry~#2%
      \fi
    }%
    \begin{entrybase}%
    \renewcommand{\theentryref}{#2}%
    \refstepcounter{entryref}%
  }
  {\end{entrybase}}

\newmuskip\pFqskip
\mathchardef\pFcomma=\mathcode`,

\newcommand{\cfplus}{\genfrac{}{}{0pt}{}{}{+}}
\newcommand{\cfminus}{\genfrac{}{}{0pt}{}{}{-}}
\newcommand{\cfdots}{\genfrac{}{}{0pt}{}{}{\cdots}}

\newcommand*{\fplus}{\cfplus}
\newcommand*{\fminus}{\cfminus}
\newcommand*{\fdots}{\cfdots}

\title[Ramanujan's elementary continued fractions]
      {An elementary treatment of Ramanujan's elementary continued fractions}

\author[G.~Bhatnagar]{Gaurav Bhatnagar
}
\address{RamanujanExplained.org, 18 Chitra Vihar, Delhi 110092,India.}
\email{bhatnagarg@gmail.com}



\begin{document}

\begin{abstract}
We study  the elementary continued fractions of Ramanujan. The  entries considered are precursors of his hypergeometric and $q$-hypergeometric continued fractions. They appear in the first third of Chapter 12 of Ramanujan's second notebook.
\end{abstract}

\maketitle
\begin{flushright}
\begin{minipage}{0.65\textwidth}
\raggedleft
\itshape
His mastery of continued fractions was, on the
formal side at any rate, beyond that of any mathematician
in the world\dots
\par\medskip
\normalfont
---Hardy~\cite[p.~XXX]{Ramanujan-CW}
\end{minipage}
\end{flushright}

\bigskip

\section{Introduction}

The objective of this paper is to study Ramanujan's early work on continued fractions.  We say it is early, because most of the entries appear in the first volume of Ramanujan's famous Notebooks \cite{RamanujanNB}, and come first in Ramanujan's own re-organization of his results on continued fractions, in Chapter 12 of his second notebook. We will find that Ramanujan's mastery of continued fractions has its roots in utter simplicity---his early continued fractions do not involve any mathematics beyond what he may have learnt as a schoolboy. 

Even so, the treatment of these continued fractions by Berndt~\cite{Berndt1989}, is anything but elementary.
In the words of Berndt and Bhargava~\cite{BB1993}, it is written for {\em highbrows}---as opposed to {\em lowbrows}---whom they address in \cite{BB1993}:
\begin{quote}
Many of Ramanujan's beautiful discoveries, however, are easily
understood, are elementary, and appeal to a wide variety of tastes.
Thus, this paper is written for \emph{lowbrows}. Only elementary
algebra is needed to prove the lion's share of theorems reported
here. 
%{\sl [A sentence deleted.]}
%Most are found in the unorganized portion of Ramanujan's
%second notebook, his third notebook, and problems that he posed
%for readers of the \emph{Journal of the Indian Mathematical
%Society}. 
%The results we describe fall under the headings of
%elementary algebra, equal sums of powers, and elementary
%number theory.
\end{quote}
%\hfill---Berndt and Bhargava~\cite{BB1993}
The purpose of this paper is to provide such an approach to Ramanujan's elementary continued fractions. 

Regarding Ramanujan's own approach, we have the words of Hardy~\cite[p.~XXXV]{Ramanujan-CW} to guide us:
\begin{quote}
He worked, far more than a majority of modern mathematicians, by induction from numerical examples. 
\end{quote}
In this, Ramanujan has been compared with Euler. However, unlike Ramanujan,  Euler left behind extensive notes to explain his thinking. As Polya~\cite[p.~90]{Polya1954I} said:
\begin{quote} Yet Euler seems to me to be almost unique in one respect: he takes pains to present the relevant inductive evidence carefully, in detail, in good order. He presents it convincingly but honestly, as a genuine scientist should do. His presentation is ``the candid exposition of ideas that led him to those discoveries'' and has a distinctive charm.
\end{quote}

In the case of Ramanujan, we do not know what led him to his discoveries. Ramanujan just listed his results in his notebooks, and left his derivations for later. But due to his unfortunate premature death,  he never got an opportunity to explain his unpublished results. In our own presentation, we have taken hints from Polya's words, to understand Ramanujan's thought process, by exploring how to motivate and discover his results by inductive reasoning. 

%Ramanujan's entries presented here all rely on an inductive approach. 
We present Ramanujan's entries as identities of terminating continued fractions: the final term is explicitly specified. As with terminating series, no question of convergence arises. When Ramanujan lists a few terms followed by $\&$c, by which he meant `etc.', or rather, `and so on', we take it to mean that the pattern continues for $n$ terms. This is the view we have taken in our statement of Ramanujan's entries. 
%At the very end, we will comment on the differences in our approach with the one presented in \cite{Berndt1989}.

How reasonable is this view? An illustration is perhaps worth a thousand words. Consider the following calculation:
\begin{multline*}
1 = \frac{2}{2}
=\frac{2}{1+1} =\frac{2}{1+\dfrac{3}{3}} 
=\frac{2}{1+\dfrac{3}{2+\dfrac{4}{4}}} 
=\cdots 
= 
\frac{2}{1}
\fplus
\frac{3}{2}
\fplus
\frac{4}{3}
\fplus
\frac{5}{4}
\fplus \genfrac{}{}{0pt}{}{}{\&\mathrm{c}.}
%&=\frac{x+1}{x+\dfrac{x+2}{x+1+\dfrac{x+3}{x+2+\dfrac{x+4}{x+4}}}} \\
%&=\frac{x+1}{x+\dfrac{x+2}{x+1+\dfrac{x+3}{x+2+\dfrac{x+4}{x+3+\dfrac{x+4}{x+4}}}}}.
\end{multline*}
We consider the continued fraction up to $n$ terms, and call it $X_n$. Then it is clear that $X_n$ is a constant sequence, converging to $1$.  This is our interpretation of his corollary to Entry 7 of Chapter 12, which comes in Part II of Ramanujan's Notebooks, by Berndt. (We have labelled it Entry II.12.7 below.) 

However, usually, the theory of continued fractions considers
sequences of convergents, given by
\[
%\frac{P_n}{Q_n}=
\frac{a_1}{b_1}
\fplus
\frac{a_2}{b_2}
\fplus\fdots\fplus
\frac{a_n}{b_n}.
\]
The infinite continued fraction is said to converge when this sequence converges. But to apply the theory of convergence for such an example seems artificial and unnecessary. 
Many of Ramanujan's examples (like the one above) are of the form
$$X_n=\frac{a_1}{b_1}\fplus\frac{a_2}{b_2}\fplus\fdots\fplus\frac{a_n}{b_n+c_n},$$
where $X_n$ is a constant.
 We simply regard it to be a constant sequence, which (of-course) converges. A similar approach proves two transformation formulas given in this chapter. 

The other type of continued fraction we have considered is a result  of the form
$$\sum_{k=1}^n A_k =  \frac{a_1}{b_1}\fplus\frac{a_2}{b_2}\fplus\fdots\fplus\frac{a_n}{b_n}.$$
Here the $n$th partial sum of a series equals the $n$th convergent of a continued fraction, and by the definition of the convergence of series and continued fractions, the continued fraction converges if and only if the series converges.
Thus no further proof is required for convergence. 

Berndt's treatment of Ramanujan's elementary continued fractions is different because he considers $\&$c (`and so on') to mean `up to infinity'. This immediately leads to questions about the convergence of these continued fractions. 

In the case of series, there are excellent reasons why this is the right way to do things. On this matter, we quote Edwards~\cite[p.~167]{Edwards1979}:
\begin{quote}
The infusion of infinite series into the ``analytic art'' of the seventeenth century raised immediate questions as to their behavior with respect to the ordinary algebraic properties of addition, subtraction, multiplication, division, and the extraction of roots. Could one properly manipulate infinite series in essentially the same ways as computations with ordinary algebraic expressions (i.e., polynomials) are carried out?
\end{quote}
But this is not needed in the case of continued fractions. Adding, subtracting etc.\ are not so easy to do in the case of continued fractions. Nevertheless, we can manipulate finite continued fractions using the usual grade school arithmetic, and this is what Ramanujan has done. 

So can we go back to a high school approach to continued fractions? This is what we have done, and we hope our treatment of Ramanujan's examples will make these intrinsically beautiful objects accessible to a wide audience. 

We should mention that most of the calculations in this paper are based on the treatment by Berndt~\cite{Berndt1989}, who has also included the proofs of many other mathematicians.  We have presented the ideas to suit our point of view. This study can be regarded as a prequel to our earlier work on Ramanujan's $q$-continued fractions~\cite{GB2014, GB2025}. An alternate view of the entries in this paper is given by Lorentzen~\cite{LL2008}. For an introduction to continued fractions, we recommend the classic Olds~\cite{Olds1963} book.

There is one notable correction in the statements of  Entries 10 and 11 (in \S\ref{sec:3}). We have followed an earlier version of these entries in \cite[Volume 1]{RamanujanNB}, in the notebooks published by the Tata Institute in 2012. This edition is based on a handwritten reproduction of Ramanujan's original notebooks, which came long after Berndt worked on \cite{Berndt1989}. Further, we have found that examining Ramanujan's earlier attempt at solving the problem considered in Entry 17 motivates the solution presented in this entry; see \S\ref{sec:4}. This entry is about an algorithm for writing a continued fraction as a (formal) power series. The remaining entries appear in \S\ref{sec:3}. We begin with some basic theorems in \S\ref{sec:2}. 

\subsection*{Notation to refer to Ramanujan's entries}
The primary sources for this paper are {\em Ramanujan's Notebooks} by  Berndt~\cite{Berndt1985, Berndt1989,Berndt1991, Berndt1994, Berndt1998}.
We use the following shorthand to refer to entries from these books. Entry m of Chapter n of Part A of Ramanujan's Notebooks is referred to as Entry A.n.m. 
So, for example, Entry II.12.7 refers to Entry 7 of Chapter 12 given in Part II of Berndt's volumes, that is, \cite{Berndt1989}. The other entries mentioned in the paragraph above this one are: II.12.10, II.12.11, and II.12.17.

\section{The basic theorem of continued fractions}\label{sec:2}
Ramanujan begins the chapter with a general theorem. While this theorem is classical, its statement and placement indicate the direction of Ramanujan's thought process. 

\begin{entry}{II.12.1}
\begin{subequations}
\begin{align}
\frac{a_1}{b_1}\fplus\frac{a_2}{b_2}\fplus \fdots\fplus\frac{a_n}{b_n}
&=
a_1\frac{N_{n-1}}{D_n}  \label{II.12.1a}\\
&=
\frac{a_1}{D_0D_1}
-\frac{a_1a_2}{D_1D_2}
+\frac{a_1a_2a_3}{D_2D_3}
-\cdots \label{II.12.1b}
\end{align}
\end{subequations}
to \(n\) terms, where, for $n\ge 2$,
\begin{subequations}\label{recurrences}
\begin{gather} 
N_{n-1}=b_nN_{n-2}+a_nN_{n-3}, \label{rec-N}\\
D_n=b_nD_{n-1}+a_nD_{n-2}; \label{rec-D} \\
\intertext{with initial values,}
N_{-1}=0, N_0=1; \label{init-N}\\
D_0=1, D_1=b_1. \label{init-D}
\end{gather}
\end{subequations}
\end{entry}
The first equality is the first thing one proves about  continued fractions, and is used to compute the convergents of the continued fraction.  Usually, the recurrence relations given in \eqref{recurrences} are written in such a manner that both the numerator and denominator satisfy the same recurrence relation with differing initial values. This can be accomplished by renaming $N_{n-1}$ to $N_n$. We will stick with Ramanujan's notation. 

The second equality \eqref{II.12.1b} follows (by telescoping) from what is called the {\em Casorati determinant}---usually the second item one proves about continued fractions. 
The fact that Ramanujan stated this in such a fashion suggests that he is interested in going from continued fractions to series and vice-versa. Further evidence for this is in the corollary that immediately follows this entry.

\begin{corentry}{II.12.1}\label{cor-II.12.a}
\begin{multline}\label{eq:cor-II.12.a}
a_1+a_2+a_3+\cdots \text{ (to $n$ terms)}\\
=
\frac{a_1}{1}
\fminus
\frac{a_2}{a_1+a_2}
\fminus
\frac{a_1a_3}{a_2+a_3}
\fminus
\frac{a_2a_4}{a_3+a_4}
\fminus
\frac{a_3a_5}{a_4+a_5}
\fminus\fdots \text{ (to $n$ terms)}. 
\end{multline}
\end{corentry}

The proof we give of Entry II.12.1 is the proof explained to me by Mourad Ismail; it is on the lines of  \cite{MI2009}.

\begin{proof}[Proof of \eqref{II.12.1a}]
We prove this by induction. Let the continued fraction be denoted $C_n$.  It is easy to verify the result for $n=1, 2$. Suppose it is true for $n\ge 2$; we show for $n+1$.
Note that
\begin{align*}
C_{n+1} &=
\frac{a_1}{b_1}\fplus\frac{a_2}{b_2}\fplus\cdots\fplus\frac{a_{n-1}}{b_{n-1}}\fplus\frac{a_n}{b_n}\fplus\frac{a_{n+1}}{b_{n+1}}\\
&= \frac{a_1}{b_1}\fplus\frac{a_2}{b_2}\fplus\cdots\fplus\frac{a_{n-1}}{b_{n-1}}\fplus\frac{a_nb_{n+1}}{b_nb_{n+1}+a_{n+1}}
\end{align*}
Now by induction, \eqref{rec-D} gives $D_{n+1}$, where now $b_n\mapsto b_nb_{n+1}+a_{n+1}$ and $a_n\mapsto a_nb_{n+1}$. Thus
\begin{align*}
D_{n+1} &= (b_nb_{n+1}+a_{n+1} )D_{n-1} + a_nb_{n+1} D_{n-2} \\
&= b_{n+1} (b_nD_{n-1} + a_nD_{n-2}) + a_{n+1}D_{n-1}\\
&= b_{n+1} D_n + a_{n+1}D_{n-1}.
\end{align*}
This verifies \eqref{rec-D} for $n+1$. 

Similarly, we can prove \eqref{rec-N} by induction.
\end{proof}
Before proving the second equality, we prove the Casorati determinant, which can be used to obtain the second equality.
\begin{proposition}
Let \(N_n\) and \(D_n\) be defined by \eqref{recurrences}. Let
\begin{subequations}\label{casorati}
\begin{equation}\label{casorati-a}
\Delta(n)
:=
N_{n-1}D_{n-1}-N_{n-2}D_n .
\end{equation}
Then, $\Delta (1)=1$, and, for \(n> 1\), 
\begin{equation}\label{casorati-b}
\Delta(n)=(-1)^{n-1}a_2a_3\cdots a_n.
\end{equation}
\end{subequations}
\end{proposition}
\begin{Remark} Note that $\Delta(n)$ is the $2\times 2$ Casorati determinant:
\[
\Delta(n) =
\begin{vmatrix}
N_{n-1} & D_n \\
N_{n-2} & D_{n-1}
\end{vmatrix}.
\]

\end{Remark}

\begin{proof}
For \(n=1\), using the initial values in \eqref{recurrences}, we get
\[
\Delta(1)
=
N_0D_0-N_{-1}D_1
=
1.
\]
For $n\geq 2$, 
we note that
$$\Delta(n) = -a_n\Delta(n-1).$$
To see this, use \eqref{rec-N} and \eqref{rec-D} to obtain
\begin{align*}
\Delta(n)
&=
N_{n-1}D_{n-1}-N_{n-2}D_n \\
&=
\bigl(b_nN_{n-2}+a_nN_{n-3}\bigr)D_{n-1}
-
N_{n-2}\bigl(b_nD_{n-1}+a_nD_{n-2}\bigr) \\
&=
-a_n\bigl(N_{n-2}D_{n-2}-N_{n-3}D_{n-1}\bigr) \\
&=
-a_n\Delta(n-1).
\end{align*}
Thus 
\[
\Delta(n)
=
(-a_n)(-a_{n-1})\cdots(-a_2)\Delta(1)
=
(-1)^{n-1}a_2a_3\cdots a_n.
\]
\end{proof}

\begin{proof}[Proof of \eqref{II.12.1b}]
Combining \eqref{casorati-a} and \eqref{casorati-b} (with \(n\) replaced by
\(k\)), dividing by \(D_{k-1}D_k\) and multiplying by \(a_1\), we obtain:
\[
a_1\left(
\frac{N_{k-1}}{D_k}
-
\frac{N_{k-2}}{D_{k-1}}
\right)
=
\frac{(-1)^{k-1}a_1a_2\cdots a_k}
     {D_{k-1}D_k}.
\]
Now sum this identity for \(k=1,2,\ldots,n\),
to obtain:
\begin{align*}
\sum_{k=1}^{n}
\frac{(-1)^{k-1}a_1a_2\cdots a_k}
     {D_{k-1}D_k}
     =
     \sum_{k=1}^{n}
a_1\left(
\frac{N_{k-1}}{D_k}
-
\frac{N_{k-2}}{D_{k-1}}
\right)
=
a_1\frac{N_{n-1}}{D_n},
\end{align*}
by telescoping.

This proves \eqref{II.12.1b}.
\end{proof}

\begin{proof}[Proof of Corollary to Entry \ref{cor-II.12.a}]
We work backwards from Entry II.12.1. Write the continued fraction in the
form
\[
\frac{A_1}{B_1}\fplus\frac{A_2}{B_2}\fplus\fdots\fplus\frac{A_n}{B_n},
\]
and let \(D_k\) be its denominators.
We want the series in \eqref{II.12.1b} to give
\[
a_1+a_2+\cdots+a_n
\]
term by term. 

The first few steps are shown in Table~\ref{tab:back-calc-cor}.  Let $t_k$ denote the $k$th term of the series \eqref{II.12.1b}. We denote by ${s}_k$ what it becomes after entering the previously found values of $A_j, D_j$, for $j<k$.  
Here
\(D_0=1\).

\begin{table}[htbp]
\centering
\renewcommand{\arraystretch}{1.8}
\[
\begin{array}{|c|c|c|c|c|c|c|}
\hline
k
& a_k
& t_k
& s_k
& A_k
& D_k
& B_k
\\ \hline
1
& a_1
& \displaystyle \frac{A_1}{D_0D_1}
& \displaystyle \frac{A_1}{D_1}
& a_1
& 1
& 1
\\
2
& a_2
& \displaystyle -\frac{A_1A_2}{D_1D_2}
& \displaystyle -\frac{a_1A_2}{D_2}
& -a_2
& a_1
& a_1+a_2
\\
3
& a_3
& \displaystyle \frac{A_1A_2A_3}{D_2D_3}
& \displaystyle -\frac{a_2A_3}{D_3}
& -a_1a_3
& a_1a_2
& a_2+a_3
\\
4
& a_4
& \displaystyle -\frac{A_1A_2A_3A_4}{D_3D_4}
& \displaystyle -\frac{a_1a_3A_4}{D_4}
& -a_2a_4
& a_1a_2a_3
& a_3+a_4
\\
5
& a_5
& \displaystyle \frac{A_1A_2A_3A_4A_5}{D_4D_5}
& \displaystyle -\frac{a_1a_2a_4A_5}{D_5}
& -a_3a_5
& a_1a_2a_3a_4
& a_4+a_5
\\
\hline
\end{array}
\]
\caption{Calculations for Corollary to II.12.1}
\label{tab:back-calc-cor}
\end{table}
The entries in the last column are obtained from \eqref{rec-D}.
For instance,
\(
D_2=B_2D_1+A_2D_0
\)
gives
\(
a_1=B_2-a_2,
\)
so \(B_2=a_1+a_2\). 

There is a change in the pattern at  \(k=3\). The choice \(A_3=-a_3\) gives \(D_3=a_2\), and we obtain
\[
a_2=B_3a_1-a_3 \implies
B_3=\frac{a_2+a_3}{a_1}.
\]
However, the choice \(A_3=-a_1a_3\) clears this denominator and gives
\(B_3=a_2+a_3\). From this point on the pattern is stable.

We find that
\begin{gather*}
D_0=1, D_1=1,  D_k=a_1a_2\cdots a_{k-1}\quad(k\geq 2);
\\
A_1=a_1, A_2=-a_2, A_k=-a_{k-2}a_k\quad(k\geq 3);\\
B_1=1, B_k=a_{k-1}+a_k\quad(k\geq 2).
\end{gather*}

Once we have the pattern, it is easy to prove it by induction.  
Given $A_k$ and $B_k$, we first verify \(D_k\) is as above. We have
\[
D_0=1,\qquad D_1=1,\qquad D_2=(a_1+a_2)-a_2=a_1.
\]
For \(k\geq 3\), \eqref{rec-D} yields
\[
D_k=(a_{k-1}+a_k)D_{k-1}-a_{k-2}a_kD_{k-2}.
\]
Assuming the pattern holds for $k-1$ and $k-2$, we see that
\begin{align*}
D_k
&=(a_{k-1}+a_k)a_1a_2\cdots a_{k-2}
-a_{k-2}a_k a_1a_2\cdots a_{k-3} \\
%&=a_1a_2\cdots a_{k-2}(a_{k-1}+a_k-a_k) \\
&=a_1a_2\cdots a_{k-1}.
\end{align*}
This shows the pattern for $D_k$. 

Note that $t_1=a_1$, $t_2=a_2$, and for $k\ge 3$:
\[
t_k=\frac{(-1)^{k-1}A_1A_2\cdots A_k}{D_{k-1}D_k}
=
\frac{
a_1a_2(a_1a_3)(a_2a_4)\cdots(a_{k-2}a_k)}
{(a_1a_2\cdots a_{k-2})(a_1a_2\cdots a_{k-1})}
=
a_k.
\]
This completes the proof.
\end{proof}

\begin{remark}[On the convergence of series in Corollary to Entry \ref{cor-II.12.a}]
The finite sum on the left-hand side is equal to the right-hand side. So the finite continued fraction on the right-hand side is the partial sum of the series 
$$\sum_i a_i.$$ The right-hand side is the $n$th convergent. 
Thus, in special cases where an infinite sum is desired,  the convergence of the continued fraction can be deduced from the convergence of the series. 

In this case, the convergence is actually the same as that of the usual convergence of continued fractions. 
\end{remark}

\section{Ramanujan's elementary continued fractions}\label{sec:3}

\subsection*{Some elementary continued fractions}
There are two types of entries in this section. The first type is on the lines of Corollary to Entry~\ref{cor-II.12.a}, where the $n$th convergent equals the $n$th partial sum of a series. Here the convergence of the continued fraction is obtained from the convergence of the series. 

The second type is an even simpler idea. Here is an example.
\[
x
=
x-a_1
+
\frac{a_1x}{x-a_2}
\fplus
\frac{a_2x}{x-a_3}
\fplus
\frac{a_{3}x}{x-a_{4}}
\fplus\fdots.
\]
Berndt has interpreted it as an infinite continued fraction
 and provided conditions under which the theorem holds as an analytic theorem. 
%
%However, Ramanujan may have just intended it to be used formally, until a pattern is found. Then, in specific cases, by taking a limiting process, we obtain an infinite continued fraction which converges to the left hand side. In our interpretation above, we have chosen to regard it as a formal identity until $n$ terms. In this case, 
%Berndt's own formal proof~\cite[p.~107]{Berndt1989} is sufficient.
The following is a direct quote (with changed equation numbering) from Berndt~\cite[p.~107]{Berndt1989} .
\begin{quote}
Quite possibly, Ramanujan attempted to prove Entry 2 by the following
nonrigorous argument. Trivially,
\begin{equation}\label{eq:II.12.2-step}
a_k
=
\frac{a_kx}{x-a_{k+1}+a_{k+1}},
\qquad k\geq 1.
\end{equation}
If we successively employ \eqref{eq:II.12.2-step} for
\(k=1,2,\ldots\), we find that
\[
a_1
=
\frac{a_1x}{x-a_2}
\fplus
\frac{a_2x}{x-a_3}
\fplus
\frac{a_3x}{x-a_4}
\fplus\cdots .
\]
Hence
\[
x
=
x-a_1
+
\frac{a_1x}{x-a_2}
\fplus
\frac{a_2x}{x-a_3}
\fplus
\frac{a_3x}{x-a_4}
\fplus\cdots .
\]
\end{quote}
We agree with Berndt's suggestion that Ramanujan may have obtained it this way. However, we do not agree that it is non-rigorous. 
Consider the following approach, that 
presents these calculations a little differently. We have interpreted this as a terminating continued fraction.
\begin{entry}{II.12.2} Let $n\ge 1$ and $x\neq 0$. Then
\begin{equation}\label{eq:II.12.2}
x
=
x-a_1
+
\frac{a_1x}{x-a_2}
\fplus
\frac{a_2x}{x-a_3}
\fplus\fdots\fplus
\frac{a_{n-1}x}{x-a_{n}}
\fplus
\frac{a_{n}x}{x-a_{n+1}+a_{n+1}}.
\end{equation}
\end{entry}
\begin{remark} We have given the second last term (with $a_{n-1}$) to show that the pattern followed by the second last term, fits the earlier terms of the continued fraction. When $n=1$, the continued fraction (what comes after the $x-a_1$) has only one term. 
%When $n=1$, we only have the last term, and $a_0:=1$.
\end{remark}

\begin{proof} Consider:
\begin{align*}
x
&=x-a_1+a_1 \\
&=x-a_1+\frac{a_1x}{x-a_2+a_2} \\
&=x-a_1+
\frac{a_1x}{x-a_2+
\dfrac{a_2x}{x-a_3+a_3}} \\
&=x-a_1+
\frac{a_1x}{x-a_2+
\dfrac{a_2x}{x-a_3+
\dfrac{a_3x}{x-a_4+a_4}}} \\
&=\cdots
%&=x-a_1+
%\frac{a_1x}{x-a_2+
%\dfrac{a_2x}{x-a_3+
%\dfrac{a_3x}{x-a_4+
%\dfrac{a_4x}{x-a_5+a_5}}}} \\
%&=x-a_1
%+
%\frac{a_1x}{x-a_2}
%\fplus
%\frac{a_2x}{x-a_3}
%\fplus
%\frac{a_3x}{x-a_4}
%\fplus
%\frac{a_4x}{x-a_5}
%\fplus\cdots .
\end{align*}
At each step, the left-hand side equals the right hand side. 
\end{proof}
It is not difficult to imagine that a bright young boy, given to algebraic manipulation, playing with the idea of continued fractions, would have thought like this.

\begin{remark}
Let $X_n$ denote the right-hand side of \eqref{eq:II.12.2}. Note that $X_n=x$ for all $n$, so is a constant sequence; thus it converges to $x$. This then gives a completely rigorous interpretation of the (infinite) continued fraction identity. Such a statement can be made for all such continued fractions below. 
\end{remark}

The following entry is of the same nature, and we treat it accordingly.

\begin{entry}{II.12.7}
If \(x\) is not a negative integer, then, for $n\ge 1$:
\begin{equation}\label{II.12.7}
1
=
\frac{x+1}{x}
\fplus
\frac{x+2}{x+1}
\fplus\fdots \fplus
\frac{x+n-1}{x+n-2}
%\fplus\cdots 
\fplus
\frac{x+n}{x+n-1 + \frac{x+n}{x+n}}
%
%\text{(to $n$ terms)}
.
\end{equation}
%Let $X_n$ denote the right-hand side. Then,
%$$\lim_{n\to\infty} X_n=1.$$
\end{entry}
\begin{proof}
%Again, Berndt has treated this analytically, and given a proof involving taking special cases of a later entry. We have interpreted it on the lines of Entry II.12.2.
We have
\begin{align*}
1
&=\frac{x+1}{x+1} \\
&=\frac{x+1}{x+\dfrac{x+2}{x+2}} \\
&=\frac{x+1}{x+\dfrac{x+2}{x+1+\dfrac{x+3}{x+3}}} \\
&=\cdots
%&=\frac{x+1}{x+\dfrac{x+2}{x+1+\dfrac{x+3}{x+2+\dfrac{x+4}{x+4}}}} \\
%&=\frac{x+1}{x+\dfrac{x+2}{x+1+\dfrac{x+3}{x+2+\dfrac{x+4}{x+3+\dfrac{x+4}{x+4}}}}}.
\end{align*}
Continuing in this manner, we see the right-hand side, denoted $X_n$, is a constant, and equal to $1$.
%taking limits as $n\to \infty$, we obtain the result.
\end{proof}
\begin{comment}
\[
1
=
\frac{x+1}{x}
\fplus
\frac{x+2}{x+1}
\fplus
\frac{x+3}{x+2}
\fplus
\frac{x+4}{x+3}
\fplus \fdots
\frac{x+n-1}{x+n-2}
\fplus
\frac{x+n}{x+n-1 + \frac{x+n}{x+n}}.
\]
\end{comment}
The special case, when $x=1$ is recorded by Ramanujan as a corollary:
\[
1
=
\frac{2}{1}
\fplus
\frac{3}{2}
\fplus
\frac{4}{3}
\fplus
\frac{5}{4}
\fplus\fdots .
\]
Again, this is an infinite continued fraction, and if one goes by the usual continued fraction meaning of convergence we have to give a proof (as given in  \cite{Berndt1989}); this  seems quite artificial in this context. We state it as a result for terminating continued fractions. 

\begin{corentry}{II.12.7}\label{II.12.7-cor} For $n\ge 1$, we have
\[
1
=
\frac{2}{1}
\fplus
\frac{3}{2}
\fplus
\frac{4}{3}
\fplus\fdots\fplus
\frac{n}{n-1}%\frac{5}{4}
\fplus 
\frac{n+1}{n + \dfrac{n+1}{n+1}}
%
%\text{(to $n$ terms)}
.
\]
\end{corentry}

Now we have an unusual theorem involving arithmetic progressions with common difference $a$. Both sides contain the arithmetic progressions of the form $x+ak$; the denominator of each term of the sum, the $a_i$'s and $b_i$'s all contain this sequence. 
\begin{entry}{II.12.8}
Let \(n\) denote a positive integer and suppose that \(x\neq -ka\),
where \(k\) is a positive integer such that \(1\leq k\leq n\). Then
\begin{multline}\label{II.12.8}
\sum_{k=1}^{n}
\frac{(-1)^{k+1}}{(x+a)(x+2a)\cdots(x+ka)}
= \\
\frac{1}{x+a}
\fplus
\frac{x+a}{x+2a-1}
\fplus
\frac{x+2a}{x+3a-1}
\fplus\fdots\fplus
\frac{x+(n-1)a}{x+na-1}.
\end{multline}
\end{entry}

\begin{proof}
This result is of the same type as the Corollary to Entry~\ref{cor-II.12.a}, and we work in the same way. We find $A_n$, $B_n$, $D_n$, and spot a pattern.
First we have to solve.
\[
\frac{1}{x+a}=\frac{A_1}{D_0D_1}.
\]
Take
\[
D_0=1, A_1=1, D_1=x+a, \text{ and } B_1=x+a.
\]
Thus the first partial quotient is
\[
\frac{A_1}{B_1}=\frac{1}{x+a}.
\]
For the second term, we want
\begin{subequations}
\begin{align}
-\frac{A_1A_2}{D_1D_2}
= -\frac{A_2}{(x+a)D_2} &=
-\frac{1}{(x+a)(x+2a)} , \label{II.12.18-step2a}
\\
D_2=B_2D_1+A_2D_0 &=B_2(x+a)+A_2 .\label{II.12.18-step2b}
\end{align}
\end{subequations}
Looking at \eqref{II.12.18-step2b}, it seems convenient to take $A_2=x+a$ to be able to factor out the term to compute $B_2$. Then, by
\eqref{II.12.18-step2a}, 
$D_2=(x+a)(x+2a)$, and
\(
B_2=x+2a-1.
\)
Thus the second partial quotient is
\[
\frac{A_2}{B_2}
=
\frac{x+a}{x+2a-1}.
\]
For the third term, we want
\begin{subequations}
\begin{align}
\frac{A_1A_2A_3}{D_2D_3}
=
 \frac{(x+a)A_3}{(x+a)(x+2a)D_3}
&=
\frac{1}{(x+a)(x+2a)(x+3a)}
\label{II.12.18-step3a}
\\
D_3=B_3D_2+A_3D_1
&=B_3(x+a)(x+2a)+A_3(x+a).
\label{II.12.18-step3b}
\end{align}
\end{subequations}
Looking at \eqref{II.12.18-step3b}, it seems convenient to take
\(A_3=x+2a\), since then
\[
A_3(x+a)=(x+2a)(x+a)=D_2,
\]
and a common factor \(D_2\) can be taken out. Then, by
\eqref{II.12.18-step3a},
\[
D_3=(x+a)(x+2a)(x+3a),
\]
and \eqref{II.12.18-step3b} gives
\(
B_3=x+3a-1.
\)
Thus the third partial quotient is
\[
\frac{A_3}{B_3}
=
\frac{x+2a}{x+3a-1}.
\]

The pattern is now clear. We want:
\[
D_n  = (x+a)(x+2a)\cdots  (x+na);\]
and
\begin{subequations}
\begin{align}
A_n &= \frac{D_{n-1}}{D_{n-2}}=x+(n-1)a; \label{12.8.An}\\
B_n & = \frac{D_{n}}{D_{n-1}}-1 = x+na-1. \label{12.8.Bn}
\end{align}
\end{subequations}
%The denominator recurrence can be written:
%$$\frac{D_n}{D_{n-1}} = B_n + \frac{A_n }{D_{n-1}/D_{n-2}},$$
%which is quite suggestive. 

We complete the proof by induction.
\end{proof}

From a continued fraction containing many arithmetic progressions, Ramanujan obtains, as a corollary, a continued fraction containing $e$. This motivates the sum on the left-hand side of Entry II.12.8. It is clear that what is required is a continued fraction where the $n$th convergent equals a partial sum of the series, because he needs to take a limit as $n\to\infty$.  
\begin{corentry}{II.12.8}\label{II.12.8-cor}
\begin{equation}\label{eq:II.12.8-cor}
\frac{1}{e-1}
=
\frac{1}{1}
\fplus
\frac{2}{2}
\fplus
\frac{3}{3}
\fplus
\cfdots .
\end{equation}
\end{corentry}
\begin{proof} We take $x=0$ and $a=1$ in \eqref{II.12.8}. Now take the limit as $n\to\infty$. The sum on the left-hand side is
$$\sum_{k=1}^\infty \frac{1}{k!}(-1)^{k+1}$$
which converges (absolutely) to $1-e^{-1}$. The right-hand side is equal to the partial sums of this series (by Entry~II.12.8), and is also the $n$th convergent of the infinite continued fraction. 
So it converges as $n\to \infty$ to $1-e^{-1}$. We obtain:
$$1-e^{-1} = \frac{1}{1}\fplus\frac{1}{1}
\fplus
\frac{2}{2}
\fplus
\frac{3}{3}
\fplus
\cfdots .
$$
The result follows by taking reciprocals and subtracting $1$. 
\end{proof}

The next continued fraction is an extension of  \eqref{II.12.7}. 
In our statement, it should be understood that only the last denominator does not follow the pattern of 
the first few terms of the continued fraction. From now on, we follow this convention for all such continued fractions.

The next entry also contains an arithmetic progression. 
\begin{entry}{II.12.9} Let $n\ge 1$. Suppose $x+ka+1\neq 0$, for $k=0, 1,2, \dots, n$ and $x+ka\neq 0$, for $k=1, 2, \dots, n.$ Then:
\begin{multline}\label{II.12.9}
\frac{x+a+1}{x+1}
=
\frac{x+a}{x-1}
\fplus
\frac{x+2a}{x+a-1}
\fplus %\fdots
\frac{x+3a}{x+2a-1}
\fplus 
\fdots\\
\fplus
%\\
% \fplus
%\frac{x+(n-1)a}
%{x+(n-2)a-1+
%\dfrac{x+(n)a+1}{x+(n-1)a+1}}\\
\frac{x+na}
{x+(n-1)a-1+
\dfrac{x+(n+1)a+1}{x+na+1}}
.
\end{multline}
%Let $X_n$ denote the right-hand side. Then,
%$$\lim_{n\to\infty} X_n=\frac{x+a+1}{x+1}.$$
\end{entry}

\begin{remark} Note that when $a=0$, Entry~\ref{II.12.9} reduces to something similar to Entry~II.12.7. Consider:
\begin{align*}
1 &=\frac{x}{x}  =\frac{x}{x-1+\dfrac{x}{x}} 
=\frac{x}{x-1+\dfrac{x}{x-1+\dfrac{x}{x}}} =\cdots.
\end{align*}
The terms in Entry~II.12.9 are suggested by
\eqref{12.8.An} and \eqref{12.8.Bn}. We interchange \(n-1\) and \(n\) in these two expressions and consider
\[
A_n=x+na,
\qquad
B_n=x+(n-1)a-1.
\]
These considerations may have motivated 
\eqref{II.12.9-idea} below, and this continued fraction.
\end{remark}

\begin{proof}
For \(k\geq1\), note that
\begin{equation}\label{II.12.9-idea}
\frac{x+ka+1}{x+(k-1)a+1}
=
\frac{x+ka}
{x+(k-1)a-1+
\dfrac{x+(k+1)a+1}{x+ka+1}}.
\end{equation}
The proof follows by taking $k=1$ and iterating.  The first three steps are:
\begin{align*}
\frac{x+a+1}{x+1}
&=
\frac{x+a}
{x-1+\dfrac{x+2a+1}{x+a+1}} \\
&=
\frac{x+a}
{x-1+
\dfrac{x+2a}
{x+a-1+\dfrac{x+3a+1}{x+2a+1}}}\\
&=
\frac{x+a}
{x-1+
\dfrac{x+2a}
{x+a-1+
\dfrac{x+3a}
{x+2a-1+\dfrac{x+4a+1}{x+3a+1}}}} &=\cdots.
\end{align*}
The pattern is clear, and we obtain \eqref{II.12.9}. 

We need $x+ka\neq 0$ for $k\geq 1$ and $x+ka+1\neq 0$ for $k\geq 0$ to avoid $0$ in the denominators. 
\begin{comment}
If the expression on the right-hand side of \eqref{II.12.9} is denoted by \(X_n\), then
\[
X_n=\frac{x+a+1}{x+1}
\]
for every \(n\), is a  is a constant sequence and
\[
\lim_{n\to\infty}X_n
=
\frac{x+a+1}{x+1}.
\]
\end{comment}
\end{proof}

\begin{Examples}
Let \(n\) be a positive integer. We have
\begin{enumerate}[(i)]
\item
\(\displaystyle
\frac{4}{3}
=
\frac{3}{1}
\fplus
\frac{4}{2}
\fplus
\frac{5}{3}
\fplus
\frac{6}{4}
\fplus\fdots\fplus
\frac{n+2}
{n+\dfrac{n+4}{n+3}}.
\)

\item
\(\displaystyle
\frac{5}{3}
=
\frac{4}{1}
\fplus
\frac{6}{3}
\fplus
\frac{8}{5}
\fplus
\frac{10}{7}
\fplus\fdots\fplus
\frac{2n+2}
{2n-1+\dfrac{2n+5}{2n+3}}.
\)
\end{enumerate}
\end{Examples}

\begin{proof}
In \eqref{II.12.9}, take \(x=2\) and \(a=1\) for part~(i), and
take \(x=2\) and \(a=2\) for part~(ii).
\end{proof}

Ramanujan has stated the above as infinite continued fractions. Some argument is needed to show convergence, if you require these results to be interpreted in the usual manner; see \cite{Berndt1989}.

We have seen two continued fractions with terms related to Arithmetic Progressions. If you wish to see continued fractions with terms related to 
Geometric Progressions (with common ratio $q$), see the $q$-continued fractions studied by the author in \cite{GB2014, GB2025}, and appearing in Berndt~\cite[Chapter 16]{Berndt1991}. These include the Rogers--Ramanujan continued fraction, which---interestingly---has its origins in the next couple of entries.

\subsection*{Two that are related to the Golden Mean}

The next two entries give some insight into the progress of Ramanujan's development of the subject, and are a result of a variation of the idea of the proof of Entry II.12.9. 

Our interpretation of the next two entries differs from that of Berndt. Berndt says~\cite[p.~116]{Berndt1989}
\begin{quote}
The interpretation of Entry 11 was made difficult because Ramanujan left most of his notation undefined.
\end{quote}
However, Ramanujan's words are clearer in the same entry in the first volume of the Notebook~\cite[Volume 1, p.~106]{RamanujanNB}. Ramanujan wrote:
\begin{quote}
\begin{enumerate}
\setcounter{enumi}{7}

\item If \(x\) is a positive integer, then
\[
x
=
\frac{1}{1-x}
\fplus
\frac{2}{2-x}
\fplus
\frac{3}{3-x}
\fplus\fdots\fplus
\frac{x}{0}
\fplus
\frac{x+1}{1}
\fplus
\frac{x+2}{2}
\fplus\fdots .
\]

\item If \(a\) is a positive integer and if
\[
\frac{N_a}{D_a}
=
\frac{n}{n-a}
\fplus
\frac{n+1}{n-a+1}
\fplus
\frac{n+2}{n-a+2}
\fplus\fdots,
\]
then
\[
\frac{N_{a+1}}{N_a}
=
n+2-a
+
\frac{a-1}{n+3-a}
\fplus
\frac{a-2}{n+4-a}
\fplus\fdots .
\]

Here we should equate the numerators and the denominators in lowest
terms. If
\(
N=\phi(n),
\)
then
\(
D=\phi(n-1).
\)
\end{enumerate}
\begin{comment}
\ \\
{Cor.}
\\
\begin{enumerate}[1.]
\item
\( \displaystyle
\frac{n^2+n+1}{n^2-n+1}
=
\frac{n}{n-3}
\fplus
\frac{n+1}{n-2}
\fplus
\frac{n+2}{n-1}
\fplus\fdots .
\)

\item
\(\displaystyle
\frac{n^3+2n+1}
{(n-1)^3+2(n-1)+1}
=
\frac{n}{n-4}
\fplus
\frac{n+1}{n-3}
\fplus
\frac{n+2}{n-2}
\fplus\fdots .
\)

\end{enumerate}
\end{comment}
\end{quote}
In the second notebook, the numbering of these entries changed to II.12.10 and II.12.11. 

We will now provide a possible approach to how Ramanujan arrived at these continued fractions, and what he may have meant by these remarks. 

We first recall the simplest continued fraction
\begin{equation}\label{gmean-cf}
\frac{1}{1}
\fplus
\frac{1}{1}
\fplus
\frac{1}{1}
\fplus\fdots.
\end{equation}
It can be generated from the recurrence
\begin{equation}\label{fib1}
F(k)=\frac{1}{1+F(k+1)},
\end{equation}
so 
$$F(1)=\frac{1}{1+F(2)} = \frac{1}{1}\fplus\frac{1}{1+F(3)} = \frac{1}{1}\fplus\frac{1}{1}\fplus\frac{1}{1+F(4)} \fdots.$$
There are multiple ways one can generalize this continued fraction. For example, one can try
$$F(k)=\frac{1}{1-x+F(k+1)},$$
but this does not give anything substantially different from \eqref{fib1}. Instead, consider:
\begin{equation}\label{II.12.10-H}
H(k)=\frac{k}{k-x+H(k+1)}.
\end{equation}
This leads to the continued fraction
\[
H(1)
=
\frac{1}{1-x}
\fplus
\frac{2}{2-x}
\fplus
\frac{3}{3-x}
\fplus\fdots.
\]
This is the first continued fraction in Ramanujan's entry (it is relabelled II.12.10). To evaluate it, we take
\[
H(k)=\frac{\phi(k)}{\phi(k-1)}.
\]
Then, from \eqref{II.12.10-H}, we obtain
\[
\frac{\phi(k)}{\phi(k-1)}
=
\frac{k}
{k-x+\dfrac{\phi(k+1)}{\phi(k)}},
\]
or,
\begin{equation}\label{II.12.10-phi}
\phi(k+1)=(x-k)\phi(k)+k\phi(k-1).
\end{equation}
Taking $k=0$, and $\phi(0)=1$, we find that
$$H(1)=\phi(1)=x.$$ 

The result we have obtained is again expressed as a terminating continued fraction. 

At present we have not explained why Ramanujan requires $x$ to be a positive integer. As of now, this is not required. In his statement, there is an $\frac{n}{0}$ in the denominator, which indicates that $x=n$. All this will be explained shortly. 

\begin{entry}{II.12.10} Let $\phi(n)$ be defined by \eqref{II.12.10-phi} with initial values $\phi(0)=1$ and $\phi(1)=x$. Then for $n\ge 1$:
\begin{equation}\label{eq:II.12.10}
x
=
\frac{1}{1-x}
\fplus
\frac{2}{2-x}
\fplus
\frac{3}{3-x}
\fplus\fdots\fplus
\frac{n}{n-x+\dfrac{\phi(n+1)}{\phi{(n)}}}.
\end{equation}
\end{entry}

\begin{remark} There are several notable features of this entry. First, it is a generalization of the simplest continued fraction, namely \eqref{gmean-cf}, that represents the Golden Mean. The celebrated Rogers--Ramanujan continued fraction is another example of this kind. We refer to Askey's remarks on the discovery of the Rogers--Ramanujan identities in \cite{Askey1989}, and the author's expanded explanation in \cite{GB2015}. Replacing $H$ by a ratio of $\phi$'s is precisely Askey's trick in \cite{Askey1989}, which he uses with similar effect.  Secondly, the continued fraction concerns a function which satisfies a three-term recurrence relation. Thirdly, this is the first entry where it appears that Ramanujan looked at the continued fraction first and then tried to find its value. 
\end{remark}

The next entry is about a shifted version of this continued fraction. %, so it seems like Ramanujan is wondering about convergence of the continued fraction. 
We relabel $x$ to $a$, and explicitly mention the dependence of $\phi(n)$ on $a$. So 
let $\phi_a(n)$ be defined by:
\begin{subequations}\label{II.12.11-phi}
\begin{gather}
\phi_a(0) :=1, \quad
\phi_a(1) :=a; \\
\phi_a(n+1)
=
(a-n)\phi_a(n)+n\phi_a(n-1).
\end{gather}
\end{subequations}

We reiterate that this is just a relabelling of $\phi(n)$ from Entry~II.12.10.
The expression $\phi_a(n)/\phi_a(n-1)$ is the tail of the continued fraction in  Entry~II.12.10 beginning
\begin{equation}\label{II.12.11-a}
\frac{\phi_a(n)}{\phi_a(n-1)} =
\frac{n}{n-a}
\fplus
\frac{n+1}{n+1-a}
\fplus
\frac{n+2}{n+2-a}
\fplus\fdots.
\end{equation}
The next entry considers $\phi_a(n)$ as a sequence indexed by $a$ (keeping $n$ constant) and provides a continued fraction for $\phi_{a+1}(n)/\phi_a(n)$.

\begin{entry}{II.12.11}
Let \(a\) be a positive integer, and define \(\phi_a(n)\) by \eqref{II.12.11-phi}.
Put
\(
N_a=\phi_a(n)\)
Then,
\begin{equation}\label{eq:II.12.11b}
\frac{N_{a+1}}{N_a}
=
n+2-a
+
\frac{a-1}{n+3-a}
\fplus
\frac{a-2}{n+4-a}
\fplus\fdots\fplus
\frac{1}{n+1}.
\end{equation}
\end{entry}
\begin{remark} Recall that the $x$  in Entry~II.12.10 has been relabelled to $a$ in Entry~II.12.11. Now it makes sense why Ramanujan wanted $x$ to be a positive integer. We note that this is a finite continued fraction. 
\end{remark}
\begin{proof} We will prove the recurrence relation
\begin{equation}\label{II.12.11-phi-a}
\phi_{a+1}(n)
=
(n+2-a)\phi_a(n)
+
(a-1)\phi_{a-1}(n).
\end{equation}
From this \eqref{eq:II.12.11b} follows by iteration.

Consider the exponential generating function of $\phi_a(n)$, defined by
\[
G_a(x)
:=
\sum_{n=0}^{\infty}\frac{ \phi_a(n)x^n}{n!}.
\]
Using \eqref{II.12.11-phi}, we find that
$$
G_a'(x)
=
aG_a(x)-xG_a'(x)+xG_a(x),
$$
or
\[
(1+x)G_a'(x)=(a+x)G_a(x).
\]
Since \(G_a(0)=1\), this differential equation can be solved to obtain
\begin{equation*}%\label{II.12.11-egf}
G_a(x)=e^x(1+x)^{a-1}.
\end{equation*}
From this expression, the following two identities are immediate:
%\begin{subequations}
\begin{gather*}
G_{a+1}(x)=(1+x)G_a(x);\\
G_a'(x)
=
G_a(x)+(a-1)G_{a-1}(x).
\end{gather*}
%\end{subequations}
The following recurrences follow by comparing coefficients.
\begin{subequations}
\begin{gather}
\label{II.12.11-shift}
\phi_{a+1}(n)
=
\phi_a(n)+n\phi_a(n-1);\\
\label{II.12.11-derivative}
\phi_a(n+1)
=
\phi_a(n)+(a-1)\phi_{a-1}(n).
\end{gather}
\end{subequations}
To obtain \eqref{II.12.11-phi-a}, we see from \eqref{II.12.11-shift} that
\begin{align*}
\phi_{a+1}(n)
&=
\phi_a(n)+n\phi_a(n-1);\\
&=\phi_a(n) + \phi_a(n+1) - (a-n)\phi_a(n) \quad\text{(using \eqref{II.12.11-phi})} \\
&= \phi_a(n) +  \big( \phi_a(n)+(a-1)\phi_{a-1}(n) \big) - (a-n)\phi_a(n),
\end{align*}
using \eqref{II.12.11-derivative}, which simplifies to \eqref{II.12.11-phi-a}.
\end{proof}

\begin{remark}
We can use \eqref{II.12.11-phi-a} to compute values of $\phi_a(n)$, provided we have some initial values. 
Since
\(
G_1(x)=e^x,
\)
we have
\(
\phi_1(n)=1.
\)
Next, taking \(a=1\) in \eqref{II.12.11-phi-a}, we obtain
\(
\phi_2(n)=(n+1)\phi_1(n)=n+1.
\)
Thus the initial values are
\[
\phi_1(n)=1,
\qquad
\phi_2(n)=n+1.
\]
\end{remark}

Ramanujan notes two corollaries of his entry. %We state finite forms.
\begin{corentry}[1]{II.12.11}
Let  \(r\) be a non-negative integer. Then
\begin{multline}\label{II.12.11-cor1} 
\frac{n^2+n+1}{n^2-n+1}
=
\frac{n}{n-3}
\fplus
\frac{n+1}{n-2}
\fplus
\frac{n+2}{n-1}
\fplus\fdots\fplus\\
\frac{n+r}
{n+r-3+
\dfrac{(n+r+1)^2+(n+r+1)+1}
      {(n+r)^2+(n+r)+1}}.
\end{multline}
\end{corentry}
\begin{corentry}[2]{II.12.11} Let  \(r\) be a non-negative integer. Then
\begin{multline}\label{II.12.11-cor2}
\frac{n^3+2n+1}
{(n-1)^3+2(n-1)+1}
=
\frac{n}{n-4}
\fplus
\frac{n+1}{n-3}
\fplus
\frac{n+2}{n-2}
\fplus\fdots\fplus\\
\frac{n+r}
{n+r-4+
\dfrac{(n+r+1)^3+2(n+r+1)+1}
      {(n+r)^3+2(n+r)+1}}.
\end{multline}
\end{corentry}

\begin{proof}
Using the initial values 
$\phi_1(n)=1,$ and $\phi_2(n)=n+1$
in \eqref{II.12.11-phi-a}, we obtain
\[
\phi_3(n)
=
n\phi_2(n)+\phi_1(n)
=
n^2+n+1.
\]
Further,
\begin{align*}
\phi_4(n)
&=
(n-1)\phi_3(n)+2\phi_2(n)\\
&=
(n-1)(n^2+n+1)+2(n+1)\\
&=
n^3+2n+1.
\end{align*}
So, we obtain:
\begin{align*}
\frac{\phi_3(n)}{\phi_3(n-1)}
&=
\frac{n^2+n+1}{n^2-n+1},\\
\intertext{and,}
\frac{\phi_4(n)}{\phi_4(n-1)}
&=
\frac{n^3+2n+1}
{(n-1)^3+2(n-1)+1}.
\end{align*}
Now writing the first $r+1$ terms of \eqref{II.12.11-a}, for $a=3, 4$, we get the results.
\end{proof}

\subsection*{Two more elementary continued fractions}
The next two entries also contain arithmetic progressions in some form or the other. Both are of the same nature as Entry II.12.2; the only difference is that the calculations involve quadratic terms (rather than linear terms). 
\begin{entry}{II.12.12}
If \(a\neq 0\) and \(x\neq -ka\), where \(k\) is a positive integer, then
\begin{equation}\label{eq:II.12.12}
1
=
\frac{x+a}{a}
\fplus
\frac{(x+a)^2-a^2}{a}
\fplus
\frac{(x+2a)^2-a^2}{a}
\fplus\fdots\fplus
\frac{(x+(n-1)a)^2-a^2}{x+na}.
\end{equation}
%Let $X_n$ denote the right-hand side. Then,
%$$\lim_{n\to\infty} X_n=1.$$

\end{entry}

\begin{proof}
Consider
\begin{align*}
1
&=
\frac{x+a}{x+a}\\
&=
\frac{x+a}
{a+\dfrac{x(x+2a)}{x+2a}}\\
&=
\frac{x+a}{a}
\fplus
\frac{(x+a)^2-a^2}{x+2a}\\
&=
\frac{x+a}{a}
\fplus
\frac{(x+a)^2-a^2}
{a+\dfrac{(x+a)(x+3a)}{x+3a}}\\
&=
\frac{x+a}{a}
\fplus
\frac{(x+a)^2-a^2}{a}
\fplus
\frac{(x+2a)^2-a^2}{x+3a}
=\cdots
%\\
%&=
%\frac{x+a}{a}
%\fplus
%\frac{(x+a)^2-a^2}{a}
%\fplus
%\frac{(x+2a)^2-a^2}
%{a+\dfrac{(x+2a)(x+4a)}{x+4a}}\\
%&=
%\frac{x+a}{a}
%\fplus
%\frac{(x+a)^2-a^2}{a}
%\fplus
%\frac{(x+2a)^2-a^2}{a}
%\fplus
%\frac{(x+3a)^2-a^2}{x+4a}\\
%&=
%\frac{x+a}{a}
%\fplus
%\frac{(x+a)^2-a^2}{a}
%\fplus
%\frac{(x+2a)^2-a^2}{a}
%\fplus
%\frac{(x+3a)^2-a^2}
%{a+\dfrac{(x+3a)(x+5a)}{x+5a}}\\
%&=
%\frac{x+a}{a}
%\fplus
%\frac{(x+a)^2-a^2}{a}
%\fplus
%\frac{(x+2a)^2-a^2}{a}
%\fplus
%\frac{(x+3a)^2-a^2}{a}
%\fplus
%\frac{(x+4a)^2-a^2}{x+5a}.
\end{align*}
At each step we use the identity:
\[
\bigl(x+(k-1)a\bigr)\bigl(x+(k+1)a\bigr)
=
(x+ka)^2-a^2.
\]
So at the \(k\)-th step, we use
\begin{multline*}
x+ka
=
a+x+(k-1)a
=
a+
\frac{\bigl(x+(k-1)a\bigr)\bigl(x+(k+1)a\bigr)}
{x+(k+1)a} \\
=
a+
\frac{(x+ka)^2-a^2}{x+(k+1)a}.
\end{multline*}
%Consequently, for every positive integer \(n\),
%So if $X_n$ denotes the right-hand side of \eqref{eq:II.12.12}, $X_n$ is a constant sequence and equals $1$ for all $n$. This completes the proof of the theorem. 
\end{proof}

\begin{entry}{II.12.13}
Let $n\ge 1$ and suppose $b+kd\neq 0$, for all $k=0, 1, 2, 3, \dots$. Then
%Let \(a\), \(b\), and \(d\) be complex numbers such that either
%5\(d\neq0\), 
%\(b\neq-kd\), where \(k\) is a nonnegative integer, and
%\[
%\Re\left(\frac{a-b}{d}\right)>0,
%\]
%or \(d\neq0\) and \(a=b\), or \(d=0\) and \(|a|<|b|\). Then, for
%\(n\geq1\),
\begin{multline}\label{eq:II.12.13}
a
=
\frac{ab}{a+b+d}
\fminus
\frac{(a+d)(b+d)}{a+b+3d} 
\fminus 
\frac{(a+2d)(b+2d)}{a+b+5d} 
\fminus\\
\fdots
\fminus
\frac{(a+(n-1)d)(b+(n-1)d)}
     {a+b+(2n-1)d}
\fminus
\frac{(a+nd)(b+nd)}
     {b+nd}.
\end{multline}
%Let \(X_n\) denote the right-hand side. Then
%\[
%\lim_{n\to\infty}X_n=a.
%\]
\end{entry}

\begin{proof}
We have
\begin{align*}
a
&=
\frac{ab}{b}\\
&=
\frac{ab}
{a+b+d-\dfrac{(a+d)(b+d)}{b+d}}
=
\frac{ab}{a+b+d}
\fminus
\frac{(a+d)(b+d)}{b+d}\\
&=
\frac{ab}{a+b+d}
\fminus
\frac{(a+d)(b+d)}
{a+b+3d-\dfrac{(a+2d)(b+2d)}{b+2d}}\\
%&=
%\frac{ab}{a+b+d}
%\fminus
%\frac{(a+d)(b+d)}{a+b+3d}
%\fminus
%\frac{(a+2d)(b+2d)}{b+2d}\\
&=
\frac{ab}{a+b+d}
\fminus
\frac{(a+d)(b+d)}{a+b+3d}
\fminus
\frac{(a+2d)(b+2d)}
{a+b+5d-\dfrac{(a+3d)(b+3d)}{b+3d}} 
%\\
%&=
%\frac{ab}{a+b+d}
%\fminus
%\frac{(a+d)(b+d)}{a+b+3d}
%\fminus
%\frac{(a+2d)(b+2d)}{a+b+5d}
%\fminus
%\frac{(a+3d)(b+3d)}{b+3d}
=\cdots .
\end{align*}

At the \(k\)-th step, we use
\begin{align*}
b+kd
&=
a+b+(2k+1)d-\bigl(a+(k+1)d\bigr)\\
&=
a+b+(2k+1)d
-
\frac{(a+(k+1)d)(b+(k+1)d)}
     {b+(k+1)d}.
\end{align*}
Iterating, we obtain \eqref{eq:II.12.13}. 

%Hence \(X_n=a\) for every \(n\ge 1\), so \((X_n)\) is a constant sequence and
%\[
%\lim_{n\to\infty}X_n=a.
%\]
\end{proof}

\subsection*{Two transformation formulas}

The next two entries are transformations of continued fractions. Since Ramanujan was an expert in series, where transformations between series abound, it would have been natural for him to consider whether continued fractions too can be transformed.

\begin{entry}{II.12.14}
If \(a_1,a_2,\ldots,a_{2n}\) and \(x\) are arbitrary complex numbers,
then
\begin{multline}\label{II.12.14}
\frac{a_1}{x}
\fplus
\frac{a_2}{1}
\fplus
\frac{a_3}{x}
\fplus
\frac{a_4}{1}
\fplus\fdots\fplus
\frac{a_{2n}}{1}
\\
=
\frac{a_1}{x+a_2}
\fminus
\frac{a_2a_3}{x+a_3+a_4}
\fminus
\frac{a_4a_5}{x+a_5+a_6}
\fminus\fdots\fminus
\frac{a_{2n-2}a_{2n-1}}
{x+a_{2n-1}+a_{2n}}.
\end{multline}
\end{entry}

\begin{proof}
Beginning with the continued fraction on the left-hand side of
\eqref{II.12.14}, we write
\begin{align*}
&\frac{a_1}{x}
\fplus
\frac{a_2}{1}
\fplus
\frac{a_3}{x}
\fplus
\frac{a_4}{1}
\fplus\fdots\\
&=
\frac{a_1}
{x+a_2-a_2+
\dfrac{a_2}
{1+\dfrac{a_3}
{x+\dfrac{a_4}{1+\fdots}}}}\\
&=
\frac{a_1}
{x+a_2+
a_2\left(
-1+
\frac{1}
{1+\dfrac{a_3}
{x+\dfrac{a_4}{1+\fdots}}}
\right)}\\
&=
\frac{a_1}
{x+a_2-
\dfrac{a_2a_3}
{x+a_3+\dfrac{a_4}{1+\fdots}}}.
\end{align*}

Applying the same step to the continued fraction beginning with
\(a_4\), we obtain
\begin{align*}
&\frac{a_1}{x+a_2}
\fminus
\frac{a_2a_3}
{x+a_3+\dfrac{a_4}{1+\dfrac{a_5}{x+\cdots}}}\\
&=
\frac{a_1}{x+a_2}
\fminus
\frac{a_2a_3}
{x+a_3+a_4-
\dfrac{a_4a_5}
{x+a_5+\dfrac{a_6}{1+\cdots}}}
.
\end{align*}
Continuing in this manner, we obtain \eqref{II.12.14}.
\end{proof}
\begin{remark} The continued fraction on the right-hand side of Entry~II.12.14 is the {\bf even part} of the continued fraction on the left. The continued fraction in
Andrews and Berndt~\cite[Entry~6.4.3, p.~162]{AB2005} is a $q$-continued fraction of this nature, and can be proved similarly. 
\end{remark}

\begin{comment}
Ramanujan stated the next entry as a transformation formula for infinite continued fractions.
\begin{equation*}%\label{II.12.15-o}
\frac{a_1+h}{1}
\fplus
\frac{a_1}{x}
\fplus
\frac{a_2+h}{1}
\fplus
\frac{a_2}{x}
\fplus\fdots
=
h+
\frac{a_1}{1}
\fplus
\frac{a_1+h}{x}
\fplus
\frac{a_2}{1}
\fplus
\frac{a_2+h}{x}
\fplus\fdots.
\end{equation*}
\end{comment}
The next entry is of the same nature as Entry~II.12.14. The proof is along the same lines.
\begin{entry}{II.12.15}
Let \(n\) be a positive integer. Then:
\begin{multline}\label{II.12.15}
\frac{a_1+h}{1}
\fplus
\frac{a_1}{x}
\fplus
\frac{a_2+h}{1}
\fplus
\frac{a_2}{x}
\fplus\fdots\fplus
\frac{a_n+h}{1}
\fplus
\frac{a_n}{x}
\\
=
h+
\frac{a_1}{1}
\fplus
\frac{a_1+h}{x}
\fplus
\frac{a_2}{1}
\fplus
\frac{a_2+h}{x}
\fplus\fdots\fplus
\frac{a_n}{1}
\fplus
\frac{a_n+h}{x-h}.
\end{multline}
\end{entry}

\begin{proof} Note that
\begin{align*}
\frac{a_1+h}{1}\fplus\frac{a_1}{x}
&=
\frac{a_1+h}{1+\dfrac{a_1}{x}}-h+h
\\
&=
h+
\frac{a_1+h-h\left(1+\dfrac{a_1}{x}\right)}
{1+\dfrac{a_1}{x}}
\\
&=
h+
\frac{a_1-\dfrac{a_1h}{x}}
{1+\dfrac{a_1}{x}}
\\
&=
h+
\frac{a_1(x-h)}{x+a_1}
\\
&=
h+
\frac{a_1(x-h)}
{x-h+a_1+h}
\\
&=
h+
\frac{a_1}
{1+\dfrac{a_1+h}{x-h}}
=
h+\frac{a_1}{1}\fplus\frac{a_1+h}{x-h}.
\end{align*}
Next consider
\begin{align*}
\frac{a_1+h}{1}
\fplus
\frac{a_1}{x}
\fplus
\frac{a_2+h}{1}
\fplus
\frac{a_2}{x}
&\equiv\frac{a_1+h}{1}\fplus\frac{a_1}{R},
\end{align*}
where 
\[
R=x+\frac{a_2+h}{1+\dfrac{a_2}{x}}.
\]
By the same steps as above, we find that this equals:
\begin{align*}
\frac{a_1+h}{1+\dfrac{a_1}{R}}
&=
h+
\frac{a_1}
{1+\dfrac{a_1+h}{R-h}} \\
&=h+
\frac{a_1}
{1+\dfrac{a_1+h}
{x+\dfrac{a_2+h}{1+\dfrac{a_2}{x}}-h} } \\
&=
h+
\frac{a_1}
{1+\dfrac{a_1+h}
{x+\dfrac{a_2}{1+\dfrac{a_2+h}{x-h}}}}
.
\end{align*}
At each step we use:
\begin{equation}\label{II.12.15-step}
\dfrac{A+h}{1+\dfrac{A}{R}}-h
=
\dfrac{A}{1+\dfrac{A+h}{R-h}}.
\end{equation}

To prove \eqref{II.12.15}, we begin with the left-hand side, and apply \eqref{II.12.15-step} repeatedly until the right-hand side is obtained. In the first step, we need to add and subtract $h$.
\end{proof}

\subsection*{One more where the partial sum equals the convergent}
In this final entry of this section, we see a sum where the denominator is quadratic in the index $k$. 
\begin{entry}{II.12.16}
Let \(N\) be a positive integer, and suppose that
\(m, n \neq -k\) for $k=1, 2, \dots, N$. 
Then
\begin{multline}\label{II.12.16}
\sum_{k=1}^{N}
\frac{(-1)^{k+1}}{(m+k)(n+k)}
=
\frac{1}{(m+1)(n+1)}
\fplus
\frac{(m+1)^2(n+1)^2}{m+n+3}
\\
\fplus
\frac{(m+2)^2(n+2)^2}{m+n+5}
\fplus\fdots\fplus
\frac{(m+N-1)^2(n+N-1)^2}
{m+n+2N-1}.
\end{multline}
The result remains valid if $N$ is replaced by $\infty$, provided $m, n \neq - k$, for all $k\in \mathbb N$. 
\end{entry}
\begin{remark} The continued fraction on the right-hand side of \eqref{II.12.16} is equal to the partial sum of a convergent series. So when $N\to\infty$, it converges, and converges to the infinite series obtained by taking $N\to\infty$ on the left-hand side. 
\end{remark}

\begin{proof}
This result is of the same type as Entry II.12.8 given in \eqref{II.12.8}, 
and we prove it in the same manner, by computing the right-hand side. Put
\(
P_k=(m+k)(n+k).
\)
We determine \(A_k\), \(B_k\), and \(D_k\), and then identify the pattern.

For the first term, we require
\[
\frac{1}{P_1}=\frac{A_1}{D_0D_1}.
\]
We take:
\(
D_0=1, A_1=1,D_1=P_1, \text{ and } B_1=P_1.
\)
Thus the first partial quotient is
\[
\frac{A_1}{B_1}
=
\frac{1}{(m+1)(n+1)}.
\]

For the second term, we require
\begin{subequations}
\begin{align}
-\frac{A_1A_2}{D_1D_2}
=
-\frac{A_2}{P_1D_2}
&=
-\frac{1}{P_2},
\label{II.12.16-step2a}\\
D_2=B_2D_1+A_2D_0
&=
B_2P_1+A_2.
\label{II.12.16-step2b}
\end{align}
\end{subequations}
It is convenient to take
\(
A_2=P_1^2,
\)
so \eqref{II.12.16-step2a} gives
\[
D_2=P_1P_2,
\]
and \eqref{II.12.16-step2b} gives
\[
B_2=P_2-P_1=m+n+3.
\]
Hence the second partial quotient is
\[
\frac{A_2}{B_2}
=
\frac{(m+1)^2(n+1)^2}{m+n+3}.
\]

For the third term, we require
\begin{subequations}
\begin{align}
\frac{A_1A_2A_3}{D_2D_3}
=
\frac{P_1^2A_3}{P_1P_2D_3}
&=
\frac{1}{P_3},
\label{II.12.16-step3a}\\
D_3=B_3D_2+A_3D_1
&=
B_3P_1P_2+A_3P_1.
\label{II.12.16-step3b}
\end{align}
\end{subequations}
It is convenient to take
\[
A_3=P_2^2,
\]
since then
\[
A_3D_1=P_2^2P_1=P_2D_2.
\]
Equation \eqref{II.12.16-step3a} gives
\[
D_3=P_1P_2P_3,
\]
and \eqref{II.12.16-step3b} gives
\[
B_3=P_3-P_2=m+n+5.
\]
Thus the third partial quotient is
\[
\frac{A_3}{B_3}
=
\frac{(m+2)^2(n+2)^2}{m+n+5}.
\]

The pattern is now clear. We take
\[
D_k=P_1P_2\cdots P_k
=
\prod_{j=1}^{k}(m+j)(n+j),
\]
and, for \(k\geq2\),
%\begin{subequations}
\begin{align*}
A_k
&=
\left(\frac{D_{k-1}}{D_{k-2}}\right)^2
=
P_{k-1}^2
=
(m+k-1)^2(n+k-1)^2,
%\label{II.12.16-Ak}
\\
B_k
&=
\frac{D_k}{D_{k-1}}
-
\frac{A_kD_{k-2}}{D_{k-1}}
=
P_k-P_{k-1}
=
m+n+2k-1.
%\label{II.12.16-Bk}
\end{align*}
%\end{subequations}
We complete the proof by induction.

\begin{comment}
Indeed,
\[
B_kD_{k-1}+A_kD_{k-2}
=
(P_k-P_{k-1})D_{k-1}
+
P_{k-1}D_{k-1}
=
P_kD_{k-1}
=
D_k.
\]
Moreover,
\[
A_1A_2\cdots A_k
=
P_1^2P_2^2\cdots P_{k-1}^2
=
D_{k-1}^2,
\]
and hence
\[
\frac{A_1A_2\cdots A_k}{D_{k-1}D_k}
=
\frac{D_{k-1}}{D_k}
=
\frac{1}{P_k}
=
\frac{1}{(m+k)(n+k)}.
\]
The result now follows from \eqref{II.12.1b}.
\end{comment}
\end{proof}

\section{From the continued fraction to a power series}\label{sec:4}
So far, Ramanujan has written many entries where the $n$th convergent of a  continued fraction is equal to the $n$th partial sum of a series. 
By contrast, in the entries of this section, Ramanujan expresses a continued fraction as a power series, where the partial sums may only be equal to some degree. Even so, in treating his example, we have indicated how we could use his ideas to determine the first few terms of a continued fraction from a given series. 

An attempt at the computation is given in Chapter 14 of Volume 1 of the Notebook~\cite[Volume 1, p.~107, Entry 15]{RamanujanNB}. Berndt~\cite[p.~412]{Berndt1994} says this is equivalent to Entry II.12.17, but perhaps it is better to consider it to be a preliminary version---perhaps an unsatisfying calculation replaced by a new entry when Ramanujan revised his notebook. The computation uses Entry~II.12.1.

The computations are in the formal power series ring $\mathbb{C}[[x]]$, where the coefficient of $x^k$, for each $k\in \mathbb N_0$, is computed by a finite number of arithmetic computations. 

\begin{theorem}Suppose
\begin{equation*}
\frac{a_1x}{b_1}
\fplus
\frac{a_2x}{b_2}
\fplus
\frac{a_3x}{b_3}
\fplus\fdots
=
T_1x-T_2x^2+T_3x^3-\cdots .
\end{equation*}
Let
\[
\frac{P_n}{b_n}
=
\frac{a_1a_2\cdots a_n}
{(b_1b_2\cdots b_n)^2},
\qquad
T_n-P_n=t_n.
\]
Then
%\begin{subequations}
\begin{gather*}
t_1=0,\quad
t_2=0,\\
T_1t_3-T_2^2=0,\\
T_2t_4-T_3^2=0,\\
T_3t_5-T_4^2
=
\frac{M^2}{P_1P_3},
\qquad
M=T_2T_4-T_3^2,\\
\intertext{and}
T_4t_6-T_5^2
=
\frac{N^2}{P_2P_4},
\quad
N=T_3T_5-T_4^2.
\end{gather*}
%\end{subequations}
\end{theorem}

\begin{proof} We first motivate $P_n$. 
Note that from \eqref{II.12.1b}, we obtain
\begin{multline*}
\frac{a_1x}{b_1}
\fplus
\frac{a_2x}{b_2}
\fplus
\frac{a_3x}{b_3}
\fplus\fdots
\frac{a_nx}{b_n}
%&=
%T_1x-T_2x^2+T_3x^3-\cdots \\
%&
\\
=
\frac{a_1x}{D_0D_1}
-\frac{a_1a_2x^2}{D_1D_2}
+\frac{a_1a_2a_3x^3}{D_2D_3}
-\cdots \qquad \text{(to $n$ terms)},
\end{multline*}
where 
\[
D_0=1,\qquad D_1=b_1,
\qquad
D_n=b_nD_{n-1}+a_nxD_{n-2}.
\]
For example,  $D_2 = b_1b_2 + a_2x$, $D_3=b_1b_2b_3 + (a_2b_3+a_3b_1)x$, and 
\[
D_4
=
b_1b_2b_3b_4
+\bigl(a_2b_3b_4+a_3b_1b_4+a_4b_1b_2\bigr)x
+a_2a_4x^2.
\]
Evidently, $D_{2k}$ and $D_{2k+1}$ are polynomials of degree at most $k$. This follows from the recurrence relation satisfied by $D_n$. Note also that setting $x=0$ in the recurrence for $D_n$, we obtain: $D_n(0)=b_nD_{n-1}(0)$, so $D_n(0)=b_1b_2\cdots b_n$.

Further observe that according to Ramanujan's entry, the $n$th convergent of the continued fraction is equal to the sum up to $n$ terms in \eqref{II.12.1b}. 

Now on to computing $T_n$. Since $D_0D_1=b_1$, a constant, and the second term has at-least a power of $2$, $T_1=a_1/b_1$. Similarly,
the second term can be computed from
\begin{multline}\label{eq:D_2}
-\frac{a_1a_2x^2}{D_1D_2} =-\frac{a_1a_2}{b_1(b_1b_2 + a_2x)}x^2
%=-\frac{a_1a_2}{b_1^2b_2}x^2\frac{1}{1+{a_2x}/{b_1b_2}} 
=-\frac{a_1a_2}{b_1^2b_2} x^2 \Big( 1-\frac{a_2}{b_1b_2}x+\frac{a_2^2}{b_1^2b_2^2}x^2
\cdots\Big)
\end{multline}
Clearly,
$T_2={a_1a_2}/{b_1^2b_2}$, since every other term is of degree $3$ and more. 
Similarly, when $n\ge 1$, it is evident that there is a term of degree $n$ from
%the \(n\)-th summand begins with
\[
(-1)^{n-1}
\frac{a_1a_2\cdots a_n}
{D_{n-1}(0)D_n(0)} x^n
=
(-1)^{n-1}
\frac{a_1a_2\cdots a_n}
{b_1^2b_2^2\cdots b_{n-1}^2b_n}x^n.
\]
This motivates Ramanujan's definition
\[
\frac{P_n}{b_n}
=
\frac{a_1a_2\cdots a_n}
{(b_1b_2\cdots b_n)^2}.
\]

But there are additional contributions to the coefficient of $x^n$, coming from earlier terms in the sum  \eqref{II.12.1b},  so Ramanujan defines the error $t_n$ by $T_n=P_n+t_n$, and computes $t_n$ for small values of $n$. We illustrate how this can be done in a couple of ways. 

For example, from \eqref{eq:D_2}, the contribution to the coefficient of \(x^3\) is
\[
t_3
=
\frac{a_1a_2^2}{b_1^3b_2^2}
=
\frac{T_2^2}{T_1},
\]
which implies
\(
T_1t_3-T_2^2=0.
\)
We can continue in the same manner for $t_4$. Here contributions to $T_4$ come from $P_4$, and from both \eqref{eq:D_2} and the term ${a_1a_2a_3x^3}/{D_2D_3}$.

Here is an alternative approach to compute $t_4$. Let $C_4$ denote the continued fraction up to the term $a_4x/b_4$. 
\begin{align*}
C_4
&=
\frac{a_1x}{b_1}
\fplus
\frac{a_2x}{b_2}
\fplus
\frac{a_3x}{b_3}
\fplus
\frac{a_4x}{b_4}
\\
&=
\frac{a_1x}
{b_1+
\dfrac{a_2x}
{b_2+
\dfrac{a_3x}
{b_3\left(1+\dfrac{a_4}{b_3b_4}x\right)}}}
\\
&=
\frac{a_1x}
{b_1+
\dfrac{a_2x}
{b_2+
\dfrac{a_3x}{b_3}
\left(
1-\frac{a_4}{b_3b_4}x+\cdots
\right)}}
\\
&=
\frac{a_1x}
{b_1+
\dfrac{a_2x}{b_2}
\left[
1-\frac{a_3}{b_2b_3}x
+
\left(
\frac{a_3^2}{b_2^2b_3^2}
+
\frac{a_3a_4}{b_2b_3^2b_4}
\right)x^2
+\cdots
\right]}
\\
&=
\frac{a_1}{b_1}x
\Bigg[
1
-\frac{a_2}{b_1b_2}x
+
\left(
\frac{a_2a_3}{b_1b_2^2b_3}
+
\frac{a_2^2}{b_1^2b_2^2}
\right)x^2
\\
&\qquad\qquad
-
\left(
\frac{a_2a_3a_4}{b_1b_2^2b_3^2b_4}
+
\frac{a_2a_3^2}{b_1b_2^3b_3^2}
+
\frac{2a_2^2a_3}{b_1^2b_2^3b_3}
+
\frac{a_2^3}{b_1^3b_2^3}
\right)x^3
+\cdots
\Bigg]
\\
&=
T_1x-T_2x^2+T_3x^3-(P_4+t_4)x^4+\cdots.
\end{align*}
One can verify that $T_2t_4-T_3^2=0$, and similarly verify the relations for $t_5$ and $t_6$ given by Ramanujan. 
\end{proof}

To summarize, Ramanujan used \eqref{II.12.1b} in the attempt to compute the power series. An important role was played by $P_n$.

In the next attempt, we take $b_k=1$ and begin with the continued fraction
$$\frac{1}{1}
\fplus
\frac{a_1x}{1}
\fplus
\frac{a_2x}{1}
\fplus\fdots
$$
so that the power series is of the form
$$1-A_1x+A_2x^2-A_3x^3+\cdots.$$
The objective is to compute the coefficients \(A_k\) in
\[
F(x)
=
\frac{1}{1}
\fplus
\frac{a_1x}{1}
\fplus
\frac{a_2x}{1}
\fplus\fdots
=
\sum_{k=0}^{\infty}(-1)^kA_kx^k,
\qquad A_0=1.
\]
This is regarded as an identity of formal power series in $\mathbb C[[x]]$.

Following Ramanujan’s notation in \eqref{II.12.1a}, let \(N_{n-1}/D_n\) denote the \(n\)th convergent.  Here \(a_0\) is \(1\). Thus
\begin{subequations}
\begin{equation}\label{ND-initial}
%N_{-1}=0, 
N_0=N_1=1,\qquad D_0=D_1=1,
\end{equation}
and, for $k\ge 2$,
\begin{equation}\label{ND-recur}
N_k=N_{k-1}+a_kxN_{k-2},
\quad
D_k=D_{k-1}+a_{k-1}xD_{k-2}.
\end{equation}
\end{subequations}
We also have
\begin{multline*}%\label{II.12.1-k-ab}
\frac{N_{k}}{D_{k+1}}=\frac{1}{1}
\fplus
\frac{a_1x}{1}
%\fplus
%\frac{a_2x}{1}
\fplus\fdots\fplus
\frac{a_kx}{1} \\
=
1-\frac{a_1x}{D_1D_2}
+\frac{a_1a_2x^2}{D_2D_3}
+\cdots
+(-1)^k\frac{a_1a_2\cdots a_kx^k}{D_kD_{k+1}}.
\end{multline*}
Note that to compute $A_1, \dots, A_n$, only the terms up to $x^n$ are required. The computation is inductive. We know $A_0=1$, and assume $A_1, A_2, \dots, A_{n-1}$ are known, and compute $A_n$. This is done by computing the coefficient of $x^n$ in $D_{n-1}F(x)$ in two ways. 

Consider, for $n\ge 2$,
\begin{multline*}
F-\frac{N_{n-2}}{D_{n-1}}
=(-1)^{n-1}
\frac{a_1\cdots a_{n-1}x^{n-1}}{D_{n-1}D_n}
+(-1)^n
\frac{a_1\cdots a_nx^n}{D_nD_{n+1}}
+O(x^{n+1}),
\end{multline*}
or,
\begin{multline}\label{DF-N}
D_{n-1}F-N_{n-2}
=(-1)^{n-1}
\frac{a_1\cdots a_{n-1}x^{n-1}}{D_n} \\
+(-1)^n
\frac{a_1\cdots a_nx^n D_{n-1}}{D_nD_{n+1}}
+O(x^{n+1}).
\end{multline}
Note that 
$$
\frac{1}{D_n}=1-(a_1+\cdots+a_{n-1})x+O(x^2).$$
This follows immediately from the recurrence relation for $D_n$. To find the coefficient of $x^n$ from the second term, we require the constant term of $D_{n-1}/(D_nD_{n+1})$, which is $1$. 
Thus the coefficient of $x^n$ in  \eqref{DF-N} is:
\begin{multline}\label{Pn-motive}
(-1)^{n} \big( a_1\cdots a_{n-1} (a_1+\cdots+a_{n-1}) + a_1\cdots a_n\big) \\
=
(-1)^{n}  a_1\cdots a_{n-1} \big(a_1+\cdots+a_{n}\big) =: (-1)^n P_n.
\end{multline}
This motivates the definition of $P_n$ in the following entry, which gives a recursive method to compute $A_n$. 
\begin{entry}{II.12.17}
Write
\begin{subequations}
\begin{equation}\label{II.12.17a}
\frac{1}{1}
\fplus
\frac{a_1x}{1}
\fplus
\frac{a_2x}{1}
\fplus
\frac{a_3x}{1}
\fplus\fdots
=
\sum_{k=0}^{\infty} A_k(-x)^k,
\end{equation}
where \(A_0=1\). Let
\begin{equation}\label{def-Pn}
P_n
:=
a_1a_2\cdots a_{n-1}
(a_1+a_2+\cdots+a_n),
\qquad n\geq1.
\end{equation}
Then
\begin{align*}
P_1&=A_1,\\
P_2&=A_2,\\
P_3&=A_3-a_1A_2,\\
P_4&=A_4-(a_1+a_2)A_3,\\
P_5&=A_5-(a_1+a_2+a_3)A_4+a_1a_3A_3,\\
P_6&=A_6-(a_1+a_2+a_3+a_4)A_5
 +(a_1a_3+a_2a_4+a_1a_4)A_4.
\end{align*}
In general, for \(n\geq1\),
\begin{equation}\label{II.12.17b}
P_n
=
\sum_{0\leq k<n/2}
(-1)^k\phi_k(n)A_{n-k},
\end{equation}
where \(\phi_0(n)=1\), and, for \(r\geq1\), \(\phi_r(n)\) is defined
recursively by
\begin{equation}\label{II.12.17c}
\phi_r(n+1)-\phi_r(n)
=
a_{n-1}\phi_{r-1}(n-1).
\end{equation}
\end{subequations}
Here we take $\phi_k(n)=0$, unless $0\le k< n/2$. 
\end{entry}

\begin{proof}
Let $F(x)$, $N_k$, $D_k$ be as above. 
We use the notation $[x^n] G(x)$ to denote the coefficient of $x^n$ of the formal power series $G(x)$. The definition for $P_n$ is motivated by
\eqref{Pn-motive}.

It is easy to verify \eqref{II.12.17b} for $n=1$, that is, $P_1=A_1$. Since
\[
\frac{N_1}{D_2} = \frac{1}{1+a_1 x} = 1 - a_1x + O(x^2),
\]
comparison with \(F(x)=1-A_1x+O(x^2)\) gives
\[
A_1=a_1=P_1.
\]

For  $n\ge 2$, we have seen that $[x^n] (D_{n-1}F-N_{n-2})$ is $(-1)^nP_n$. 

We first note that the degree of $N_{n-2}$ is less than $n$, for $n\ge 2$. 
This follows immediately from the recurrence relations of $N_{n-1}$. Thus, $N_{n-2}$ does not contribute anything to the coefficient of $x^n$, and so
$$[x^n] D_{n-1}F(x) = (-1)^nP_n.$$

To prove \eqref{II.12.17b}, we compute the same coefficient in another way. 

Define $\phi_k(n)$ for $n\ge 1$, from
$$D_{n-1} = \sum_{k} \phi_{k}(n) x^k,$$
where the coefficients $\phi_{k}(n)$ are taken to be $0$ when $k<0$ or when $k>\deg D_{n-1}$. Note that $\phi_0(n)=1$, and,
\[
\deg D_{n-1}\le\left\lfloor\frac{n-1}{2}\right\rfloor < \frac{n}{2}.
\]

For $n>1$, we have
\[
[x^n]\bigl(D_{n-1}F(x)\bigr)
=
\sum_{0\le k<n/2}
\phi_k(n)(-1)^{n-k}A_{n-k}.
\]
Equating this with \((-1)^nP_n\), we obtain \eqref{II.12.17b}.

Finally, the recurrence for $D_n$ in \eqref{ND-recur} implies, on comparing coefficients of $x^k$, that  for \(n\ge2\) and \(k\ge1\),
\[
\phi_k(n+1)-\phi_k(n)
=a_{n-1}\phi_{k-1}(n-1).
\]
Note the initial condition $\phi_0(n)=1$, and the boundary conditions $\phi_k(n)=0$ unless $k<n/2$. With these conditions, this recurrence works for $n=1$ too, and determines $\phi_k(n)$ for $n\ge 1$, $k\ge 1$. 
This completes the proof of \eqref{II.12.17c}, and the entry.
\end{proof}

\begin{remark}
Ramanujan has the following as Corollary (ii). We quote from \cite[Volume 2, p.~141]{RamanujanNB}:
\begin{quote}
In the above results:
$$D_{r-1} = \phi_0(r)+x\phi_1(r)+x^2\phi_2(r)+\cdots.$$
\end{quote}
This is evidence that our exposition is consistent with Ramanujan's own thought process.
\end{remark}

\begin{remark} We can reverse-engineer Entry II.12.17 to obtain an algorithm to expand a power series as a continued fraction. The algorithm obtained is similar to the one given by Frank~\cite[Theorem 2.1]{Frank1946}.

Let $$F(x) = 1-A_1x+A_2 x^2 +\cdots,$$
where $A_1\neq 0$. Take $a_1=A_1$ and $D_0=1=D_1$. For $n\ge 2$, 
suppose that \(a_1,\ldots,a_{n-1}\) have already been found and are non-zero. Then, to find $a_n$, proceed as follows. 
\begin{enumerate}
\item
Compute
\(
P_n=(-1)^n[x^n]\bigl(D_{n-1}F(x)\bigr).
\)
\item Set
\[
a_n=\frac{P_n}{a_1\cdots a_{n-1}}
-(a_1+\cdots+a_{n-1}).
\]
\item If $a_n\neq 0$, 
find \(D_{n}\) using the recurrence relation $$D_n=D_{n-1}+a_{n-1}xD_{n-2},$$ for the next iteration. 
\item  If $a_n=0$, the algorithm breaks down and is terminated.
\end{enumerate}
If $a_1, a_2, \dots, a_n \neq 0$, we obtain the terms of a continued fraction 
\[
F(x)
=
\frac{1}{1}
\fplus
\frac{a_1x}{1}
\fplus
\frac{a_2x}{1}
\fplus\fdots\fplus
\frac{a_nx}{1}
+O(x^{n+1}).
\]
\end{remark}
We illustrate this remark by working out Ramanujan's example. 
\begin{Example}
We have
\begin{multline}\label{ex:4.5}
%\bigg({}_2F_1\Big(\frac12,\frac12;1;x\Big)\bigg)^2
%=
\bigg(
1+\Big(\frac12\Big)^2x
+\Big(\frac{1\cdot3}{2\cdot4}\Big)^2x^2
+\Big(\frac{1\cdot3\cdot5}{2\cdot4\cdot6}\Big)^2x^3
+\Big(\frac{1\cdot3\cdot5\cdot7}{2\cdot4\cdot6\cdot8}\Big)^2x^4
+\cdots
\bigg)^2\\
=
\frac{1}{1}
\fminus
\frac{x}{2}
\fminus
\frac{3x}{8}
\fminus
\frac{5x}{2}
\fminus
\frac{17x}{40}
\fminus
\frac{23x}{2}
\fminus
\frac{1395x}{3128}
\fminus\fdots.
\end{multline}
\end{Example}
While Ramanujan may have given this as a means to verify Entry II.12.17, we use it to illustrate the above remark, and compute the first few values of $a_i$ from the series. 

Squaring the series (formally) gives
\begin{align*}
F(x)
&=1+\frac12x+\frac{11}{32}x^2+\frac{17}{64}x^3
+\frac{1787}{8192}x^4
+\frac{3047}{16384}x^5
+\frac{42631}{262144}x^6+O(x^7).
\\
&\equiv 1-A_1x+A_2x^2 -A_3 x^3+\cdots
\end{align*}
%We see that
%Since \(F(x)=\sum_{n\ge0}A_n(-x)^n\), we have
%\[
%A_1=-\frac12,\quad A_2=\frac{11}{32},\quad
%A_3=-\frac{17}{64}
%\quad A_4=\frac{1787}{8192},
%\quad A_5=-\frac{3047}{16384},\quad
%A_6=\frac{42631}{262144}
%.
%\]

We begin with \(D_0=D_1=1\) and
\(
a_1=A_1=- 1/2.
\)

Since \(D_1=1\), the next step gives
\[
P_2=[x^2](1\cdot F(x)) = A_2=\frac{11}{32};
\quad
a_2=\frac{11/32}{-1/2}+\frac12=-\frac3{16}.
\]
Using \(a_1\), we find $D_2$:
\[
D_2=D_1+a_1xD_0=1-\frac12x.
\]
Consequently,
\[
P_3=-[x^3]\bigl(D_2F(x)\bigr)
=-\left(\frac{17}{64}-\frac12\frac{11}{32}\right)
=-\frac3{32},
\]
and
\[
a_3=\frac{-3/32}{3/32}+\frac{11}{16}
=-\frac5{16}.
\]
The next 3 steps yield
\[
a_4= -\frac{17}{80}, \quad
a_5 =-\frac{23}{80},\quad
a_6 =-\frac{1395}{6256}.
\]

The continued fraction obtained is equivalent to the one in \eqref{ex:4.5}.  
This can be seen by the following calculation. 
\begin{multline*}
\frac{1}{1}
\fminus\frac{x/2}{1}
\fminus\frac{3x/16}{1}
\fminus\frac{5x/16}{1} 
=
\frac{1}{1}
\fminus\frac{x}{2}
\fminus\frac{3x/8}{1}
\fminus\frac{5x/16}{1}\\
=
\frac{1}{1}
\fminus\frac{x}{2}
\fminus\frac{3x}{8}
\fminus\frac{5x/2}{1}
=
\frac{1}{1}
\fminus\frac{x}{2}
\fminus\frac{3x}{8}
\fminus\frac{5x}{2}.
\end{multline*}
The above shows how we get
an equivalent continued fraction, by absorbing the denominators in the next denominator of the continued fraction. 

The next corollary is actually a more general version of Entry II.12.17. Ramanujan possibly labelled it as a corollary, because it can be proved in the same manner. 
\begin{corentry}[i]{II.12.17}
Write
\[
\frac{1}{1+b_1x}
\fplus
\frac{a_1x}{1+b_2x}
\fplus
\frac{a_2x}{1+b_3x}
\fplus\fdots
=
\sum_{k=0}^{\infty} A_k(-x)^k,
\]
where \(A_0=1\). Define
\[
P_n
=
a_1a_2\cdots a_{n-1}
(a_1+b_1+a_2+b_2+\cdots+a_n+b_n),
\qquad n\ge1.
\]
Then, for \(n\ge1\),
\[
P_n
=
\sum_{k=0}^{n-1}(-1)^k\phi_k(n)A_{n-k},
\]
where \(\phi_k(n)\), \(n, k\ge1\), is defined
recursively by
\[
\phi_k(n+1)-\phi_k(n)
=
b_n\phi_{k-1}(n)
+a_{n-1}\phi_{k-1}(n-1),
\]
where we have the initial and boundary conditions:
$$\phi_0(0)=0,  \quad \phi_0(n)=1 \text{ for $n\ge 1$}, \quad \phi_k(n)=0 \text{ for $k\ge n$}.$$
\end{corentry}

\begin{proof}
The proof is on the lines of the proof of Entry II.12.17. 
Let
\[
F(x)
=
\frac{1}{1+b_1x}
\fplus
\frac{a_1x}{1+b_2x}
\fplus
\frac{a_2x}{1+b_3x}
\fplus\fdots
=
\sum_{k=0}^{\infty}A_k(-x)^k.
\]
As before, write the \(n\)th convergent as \(N_{n-1}/D_n\).
The numerators and denominators satisfy
\[
N_0=1,\qquad N_1=1+b_2x,
\qquad
D_0=1,\qquad D_1=1+b_1x,
\]
and, for \(n\ge2\),
\[
\begin{aligned}
N_n&=(1+b_{n+1}x)N_{n-1}+a_nxN_{n-2},\\
D_n&=(1+b_nx)D_{n-1}+a_{n-1}xD_{n-2}.
\end{aligned}
\]
These recurrences give
\[
\deg N_{n}\le n,\qquad \deg D_n\le n.
\]
If all \(b_i\neq0\), the degrees are exactly $n$.

We see that \(D_n(0)=1\), and, by induction,
\[
D_n
=
1+\bigl(b_1+\cdots+b_n+a_1+\cdots+a_{n-1}\bigr)x
+O(x^2).
\]

As before, it is easy to see that  \(A_1=a_1+b_1=P_1\).

For \(n\ge2\), Entry~II.12.1 gives, just as before,
\begin{align*}
D_{n-1}F-N_{n-2}
={}&(-1)^{n-1}
\frac{a_1\cdots a_{n-1}x^{n-1}}{D_n}\\
&+(-1)^n
\frac{a_1\cdots a_nD_{n-1}x^n}{D_nD_{n+1}}
+O(x^{n+1}).
\end{align*}
The coefficient of \(x^n\) on the right is
\begin{align*}
&(-1)^n a_1\cdots a_{n-1}
\bigl(b_1+\cdots+b_n+a_1+\cdots+a_{n-1}\bigr)
+(-1)^n a_1\cdots a_n\\
&\qquad=
(-1)^n a_1\cdots a_{n-1}
(a_1+b_1+\cdots+a_n+b_n)
=
(-1)^nP_n.
\end{align*}
Again, 
\(\deg N_{n-2}\le n-2\), so
\[
[x^n]\bigl(D_{n-1}F\bigr)=(-1)^nP_n.
\]

Now define \(\phi_k(n)\) by
\[
D_{n-1}(x)=\sum_{k=0}^{n-1}\phi_k(n)x^k.
\]
We have \(\phi_0(n)=1\), and we take $\phi_k(n)=0$ unless $0\le k\le n-1$. 
%This is possible since \(\deg D_{n-1}\le n-1\).
%, and coefficients outside the
%indicated range are taken to be zero. 
As earlier, we get 
\[
P_n= (-1)^n [x^n]\bigl(D_{n-1}F\bigr)
=
\sum_{k=0}^{n-1}
\phi_k(n)(-1)^{k}A_{n-k}.
\]

Finally, the denominator recurrence implies
\[
D_n-D_{n-1}
=
b_nxD_{n-1}+a_{n-1}xD_{n-2}.
\]
Comparing coefficients of \(x^k\), for \(n\ge2\) and \(k\ge1\),
gives
\[
\phi_k(n+1)-\phi_k(n)
=
b_n\phi_{k-1}(n)
+a_{n-1}\phi_{k-1}(n-1).
\]
Since \(D_0=1\) and \(D_1=1+b_1x\), we have
\[
\phi_k(2)-\phi_k(1)
=
\begin{cases}
b_1,&k=1,\\
0,&k\ge2.
\end{cases}
\]
Thus for the recurrence to hold for \(n=1, k=1\), we need $\phi_0(0)=0$, in addition to the previously stated conditions:
\(\phi_0(n)=1\) for \(n\ge1\), and \(\phi_k(n)=0\) for \(k\ge n\). With these initial and boundary conditions, we see this recurrence determines all the coefficients. 
\end{proof}

Before concluding, we note that if we replace $a_kx$ by $a x^k$ in the continued fraction considered in \eqref{II.12.17a}, we get the Rogers--Ramanujan continued fraction:
\[
\frac{1}{1}
\fplus
\frac{a x}{1}
\fplus
\frac{ax^2}{1}
\fplus\frac{ax^3}{1}
\fplus\fdots.
\]
A similar comment can be made about the continued fraction in Corollary (i) to Entry II.12.17, which motivates some of Ramanujan's generalizations of the Rogers--Ramanujan continued fraction. 

\section{What is the root of Ramanujan's genius?}
This study hints at some of the properties of Ramanujan, the mathematician. He did what many mathematicians do---repeatedly use the same trick until it becomes a technique, and organize his results from easy to complicated. He considered series and continued fractions with arithmetic progressions because special and limiting cases yield the exponential function. He asked himself questions---for example, how does one write a series corresponding to a continued fraction?---and then answered them, refining his answers over time, as he got more ideas. Such simple questions---what happens when one modifies the simplest continued fraction containing all $1$s, or when one replaces an arithmetic progression with a geometric progression---eventually led to the celebrated Rogers--Ramanujan identities. 

Perhaps Ramanujan's genius lay in his simplicity. We just have to find the right way to look at his results to understand his genius.

\section*{Declaration of the use of AI}
We have used GPT 5.6 Sol and Astra (made by OpenAI) while writing this paper. These AI engines have been used to write SageMath code, to perform symbolic computations, and in generating \LaTeX. They were useful in generating some proofs too. Once we got an idea of a proof we wanted to use, and wrote it down, we could check if the same kind of proof worked on another identity, and then generate an analogous proof. This speeded up the work considerably. The author takes responsibility for the contents of this paper.

%\bibliography{$HOME/Documents/Papers/references}{}
%\bibliographystyle{abbrv}
%\bibliographystyle{alpha}

\end{document}